\documentclass[11pt]{amsart} 
\usepackage{latexsym}
\usepackage{amsfonts}
\usepackage{amsmath,amsthm,amssymb}
\usepackage{fullpage}
\usepackage{mathabx}
\usepackage{tikz}
\usetikzlibrary{matrix}
\usepackage[colorlinks=true, linkcolor=blue, citecolor=magenta]{hyperref}
\usepackage{mathrsfs}
\usepackage{soul}
\usepackage{arydshln}

\newcommand{\cal}{\mathcal}

\def\epsilon{\varepsilon}
\def\phi{\varphi}

\def\subset{\subseteq}

\newcommand{\card}{\mbox{card}}

\newcommand{\bvec}{\overleftarrow}

\newcommand{\R}{\mathbb R}
\newcommand{\Z}{\mathbb Z}

\def\strutdepth{\dp\strutbox}
\def \ss{\strut\vadjust{\kern-\strutdepth \sss}}
\def \sss{\vtop to \strutdepth{
\baselineskip\strutdepth\vss\llap{$\diamondsuit\;\;$}\null}}

\def\strutdepth{\dp\strutbox}
\def \sst{\strut\vadjust{\kern-\strutdepth \ssss}}
\def \ssss{\vtop to \strutdepth{
\baselineskip\strutdepth\vss\llap{$\spadesuit\;\;$}\null}}

\def\strutdepth{\dp\strutbox}
\def \ssh{\strut\vadjust{\kern-\strutdepth \sssh}}
\def \sssh{\vtop to \strutdepth{
\baselineskip\strutdepth\vss\llap{$\heartsuit\;\;$}\null}}

\def\qed{\hfill\rlap{$\sqcup$}$\sqcap$\par}

\def\strutdepth{\dp\strutbox}
\def \ss{\strut\vadjust{\kern-\strutdepth \sss}}
\def \sss{\vtop to \strutdepth{
\baselineskip\strutdepth\vss\llap{$\diamondsuit\;\;$}\null}}

\def\strutdepth{\dp\strutbox}
\def \sst{\strut\vadjust{\kern-\strutdepth \ssss}}
\def \ssss{\vtop to \strutdepth{
\baselineskip\strutdepth\vss\llap{$\spadesuit\;\;$}\null}}
\def\qed{\hfill\rlap{$\sqcup$}$\sqcap$\par}

\newtheorem{thm}{Theorem}%[section]
\newtheorem{cor}[thm]{Corollary}

\newtheorem{prop}[thm]{Proposition}

\theoremstyle{definition}

\newtheorem{rem}[thm]{Remark}
\newtheorem{defn-rem}[thm]{Definition-Remark}

\theoremstyle{remark}

\numberwithin{equation}{section}

\begin{document}

\author[N.~Bedaride]{Nicolas B\'edaride}
\author[A.~Hilion]{Arnaud Hilion}
\author[M.~Lustig]{Martin Lustig}

\address{\tt 
Aix Marseille Universit\'e, CNRS, Centrale Marseille, I2M UMR 7373,
13453  Marseille, France}
\email{\tt Nicolas.Bedaride@univ-amu.fr}
\email{\tt Arnaud.Hilion@univ-amu.fr}
\email{\tt Martin.Lustig@univ-amu.fr}

\title[Measures on subshifts]{Invariant measures on finite rank subshifts}
 
\begin{abstract} 
In this note we show that for any subshift $X$ of finite $S$-rank every invariant measure $\mu$ is determined by its values on finitely many cylinders. Under mild conditions these cylinders are given by the letters of the alphabet in question.
\end{abstract}

%{Primary 20F36, Secondary 20E36, 57M05} 
%\subjclass[2000]{Primary 20F, Secondary 20E, 57M}
%, 37B, 37D}
 
\maketitle 

For any finite discrete set $\cal A = \{a_1, \ldots, a_d\}$ the infinite product $\cal A^\Z$ is compact, with respect to the product topology, and indeed it is a Cantor set. It is naturally equipped with a {\em shift map} $T$, induced by subtracting 1 from the coordinate indices of $\cal A^\Z$, and $T$ acts as homeomorphism on $\cal A^\Z$. A subset $X \subset \cal A^\Z$ is called a {\em subshift} (over $\cal A$) if $X$ is closed and if $T(X) = X$.

An {\em invariant measure} $\mu$ on $X$ is a finite $T$-invariant Borel measure on $\cal A^\Z$ with support contained in $X$. Any such $\mu$ is determined by its values $\mu([w])$ on the family of {\em cylinders} $[w]$, for all words $w = y_1 y_2 \ldots y_n$ (with $y_i \in \cal A$) in the free monoid $\cal A^*$ over the above {\em alphabet} $\cal A$. Here $[w]$ consists of all biinfinite words $\ldots x_{j-1} x_j x_{j+1} \ldots  \in \cal A^\Z$ with $x_i = y_i$ for all $i = 1, \ldots, n$. The integer $n \geq 0$ is called the {\em length} of $w$ and is denoted by $|w|$.

For any subshift $X \subset \cal A^\Z$ we consider the linear surjective map $\zeta$ from the cone $\cal M(X)$ of invariant measures on $X$ onto cone of {\em letter measures} $\cal C(X) \subset \R^\cal A$, given by $\mu \mapsto (\mu([a_1]), \ldots, \mu([a_d]))$. 

\smallskip

A very convenient tool to define a subshift $X$ is a {\em directive sequence} ${\bvec \sigma} = (\sigma_n)_{n \geq 0}$ of monoid morphisms $\sigma_n: \cal A^*_{n+1} \to \cal A^*_n$, with $\cal A_n = \{a_{1, n}, \ldots, a_{d(n), n}\}$ for all $n \geq 1$ and $\cal A_0 = \cal A$. Here the subshift $X$ {\em generated by} $\bvec \sigma$ consists of all biinfinite words $x = \ldots x_{n-1} x_n x_{n+1} \ldots  \in \cal A^\Z$ such that every finite factor $x_j x_{j+1} \ldots x_k$ of $x$ occurs also as factor in some $\sigma_0 \circ \ldots \circ \sigma_{n-1}(a_{k, n})$. 

The directive sequence $\bvec \sigma$ is also called an {\em $S$-adic development} of $X$ if $S$ is a (possibly infinite) set of monoid morphisms which contains all $\sigma_n$ that occur in $\bvec \sigma$. The set $S$ is of particular interest if the cardinality of all {\em level alphabets} $\cal A_n$ is bounded; in this case we can identify all level alphabets $\cal A_n$ with a single alphabet $\cal A$.

A directive sequence ${\bvec \sigma} = (\sigma_n)_{n \geq 0}$ is called {\em everywhere growing} if the sequence of {\em minimal lengths} $\beta_-(n) = \min\{|\sigma_0 \circ \ldots \circ \sigma_{n-1}(a_{k, n})| \mid a_{k, n} \in \cal A_n\}$ tends to $\infty$ for increasing $n$. Every subshift $X$ possesses an everywhere growing $S$-adic development (see Monteil's Theorem 4.3 of \cite{BD} or Proposition 2.7 of \cite{BHL1}), and it makes sense to think of $X$ as ``small'' if its possesses a directive sequence which is both, everywhere growing and of bounded cardinalities $\card (\cal A_n)$. We say that such a subshift $X$ has {\em finite $S$-rank} (since it seems closely related to what has been called ``finite (topological) rank'' in the context of Bratteli-Vershik systems, see for instance \cite{BD}). 

To think of $X$ as ``small'' 
is justified for instance by the well known fact 
that for any subshift $X$ of finite $S$-rank the cone $\cal M(X)$ is of finite dimension. Indeed, this dimension is equal to 
the number $e(X)$ of invariant ergodic probability measures on $X$, 
and (see for instance Remark 4.3 of \cite{DLR}) $e(X)$ is bounded above by the cardinality $\card (\cal A)$ of the alphabet $\cal A$ which is used by any everywhere growing $S$-adic development ${\bvec \sigma} = (\sigma_n)_{n \geq 0}$ of $X$. Hence one has 
$$c(X) \leq e(X) \leq \card(\cal A) \, ,$$ 
where we denote by $c(X)$ the dimension of the above cone $\cal C(X)$ of  letter measures.

For any subshift $X$, given as before by a directive sequence $\bvec \sigma = (\sigma_m)_{m \geq 0}$, we will also consider, for any $n \geq 0$, the dimension $c_n$ of the letter cone $\cal C(X_n)$ of the {\em intermediate level-$n$ subshift} $X_n$. The latter is defined by truncating the given directive sequence $\bvec \sigma$ at level $n$; in other words, $X_n$ is the subshift generated by the directive sequence $(\sigma_m)_{m \geq n}$. In the special (but rather frequently occurring) case where $c_n = c(X)$ for all $n \geq 0$, the directive sequence $\bvec \sigma$ is called {\em thin}. We say that a subshift $X$ is {\em thin} if $X$ possesses an $S$-adic development which is everywhere growing, has bounded level alphabet cardinalities, and is thin.

\begin{thm}
\label{main}
Let $X$ be a thin subshift over a finite alphabet $\cal A$. Then any two invariant measures $\mu_1$ and $\mu_2$ on $X$ are equal if and only if one has $\mu_1([a_k]) = \mu_2([a_k])$ for the finitely many cylinders $[a_k]$ given by all $a_k \in \cal A$.
\end{thm}

Any monoid morphism $\sigma: \cal A^* \to \cal {A'}^*$, for finite alphabets $\cal A = \{a_1, \ldots, a_d\}$ and $\cal {A'} = \{a'_1, \ldots, a'_{d'}\}$ defines an {\em incidence matrix} $M(\sigma) = (|\sigma(a_j)|_{a'_i})_{i, j}$ of seize $d' \times d$, where $|w|_x$ denotes the number of occurrences of the letter $x$ in the word $w$. Since $M(\sigma)$ describes also the linear map $\R^d \to \R^{d'}$ on the associated vector spaces that is given by the abelianization of monoid morphism $\sigma$, one derives directly from Theorem \ref{main}:

\begin{cor}
\label{cor}
Let $\cal A$ be a finite alphabet, and for any $n \geq 0$ let $\sigma_n: \cal A \to \cal A$ be a monoid morphism with invertible incidence matrix $M(\sigma_n)$. If the directive sequence $(\sigma_n)_{n \geq 0}$ is  everywhere growing, then any invariant measure $\mu$ on the subshift $X$ generated by this sequence is determined by the evaluation of $\mu$ on the letter cylinders, i.e. by the values $\mu([a_k])$ for all $a_k \in \cal A$.
\end {cor}

The conclusion of Corollary \ref{cor} has recently been proved by Berth\'e et al. 
under distinctly more restrictive hypotheses (see Corollary 4.2 of \cite{B+Co}). This sparked among the authors the feeling that the results of our papers \cite{BHL1} and \cite{BHL2} 
are yet not well enough divulged in the symbolic dynamics community. We hence decided to write this note in order to show how the above statements follow as directs consequences of our previous results.

\begin{proof}[Proof of Theorem \ref{main}]
For any $S$-adic development $\bvec \sigma = (\sigma_n: \cal A^*_{n+1} \to \cal A^*_n)_{n \geq 0}$ of $X$ let $\zeta_n: \cal M(X) \to C(X_n) =: \cal C_n \subset \R^{\cal A_n}$ be the map for the level $n$ subshift $X_n$ analogue to the above map $\zeta$ for $X$. The main result of \cite{BHL1} (see also Theorem 3.2 of \cite{BHL2}) shows that for any everywhere growing $\bvec \sigma$ and any level $n \geq 0$ the linear map $M(\sigma_n)$ maps $\cal C_{n+1}$ onto $\cal C_n$, and that for the resulting inverse limit cone $\cal V$ of the $C_n$ there is a surjective linear map $\frak m: \cal V \to \cal M(X)$ which splits over the maps $\zeta_n$. 

In Corollary 10.4 of \cite{BHL1} it is shown that, if all $\cal C_n$ have the same finite dimension $c_n = c(X)$, then all of the maps $\zeta_n$ are bijective. For $n = 0$ this gives the claimed statement.
\end{proof}

\begin{rem}
\label{new-1}
The proof (and hence the statement of Theorem \ref{main}) is correct for the more general case where the $S$-adic development is not of bounded level alphabet cardinality, as long as it is everywhere growing and all $c_n$ are equal.
\end{rem}

\begin{proof}[Proof of Corollary \ref{cor}]
From the hypothesis that $M(\sigma_n)$ is invertible it follows directly that $c_{n+1}= c_n$, so that the hypotheses of Theorem \ref{main} are satisfied.
\end{proof}

We finish this note with the observation that for any subshift $X$ of finite $S$-rank and any everywhere growing $S$-adic development of $X$ over $\cal A$ it follows directly from the finiteness of $\cal A$ that there is a level $n_0 \geq 0$ with $c_n = c_{n_0}$ for all $n \geq n_0$. It follows from Theorem \ref{main} that any invariant measure $\mu_{n_0}$ on the intermediate subshift $X_{n_0}$ is determined by the values $\mu_{n_0}([a_k])$, for any $a_k \in \cal A$. Thus we can apply Proposition \ref{measure-surj} below to conclude: 

\smallskip

{\em Any invariant measure $\mu$ on $X$ is determined by the values $\mu([\sigma_0 \circ \ldots \circ \sigma_{n_0 - 1}(a_k)])$ for all $a_k \in \cal A$. }

\smallskip

Although we don't know a precise reference, the following basic fact is probably well known to the experts:

\begin{prop}
\label{measure-surj}
Let $\sigma: \cal A^* \to \cal B^*$ be a non-erasing monoid morphism (i.e. $\sigma(a_k)$ is non-trivial for any $a_k \in \cal A$), and let $X$ be any subshift over $\cal A$. Let $Y$ be the subshift over $\cal B$ generated by the set $\sigma(X)$. Then one has:

\begin{enumerate}
\item
The map $\sigma$ induces a well defined map $\sigma_\cal M: \cal M(X) \to \cal M(Y)$.
\item
The map $\sigma_\cal M$ is surjective.
\item
The map $\sigma_\cal M$ is functorial: 
For any second non-erasing monoid morphism $\sigma'$ one has $(\sigma \sigma')_\cal M = \sigma_\cal M \sigma'_\cal M$.
\qed
\end{enumerate}
\end{prop}

\begin{rem}
\label{new-2}
(1)
The map $\sigma_\cal M$ isn't quite as natural as one may wish it to be: If $\mu$ is a probability measure (i.e. $\mu(X) = 1$), then the image measure $\sigma_\cal M(\mu)$ need not be a probability measure.

\smallskip
\noindent
(2)
A direct proof of Proposition \ref{measure-surj} follows from the above cited surjectivity of the map $\frak m$, since any everywhere growing directive sequence for $X$ is prolongated by $\sigma$ to an everywhere growing directive sequence for $Y$,
as long as $\sigma$ is non-erasing.

An alternative (but more elaborate) proof can be derived by elementary methods, which does not depend on the particular directive sequence chosen to generate the given subshift $X$.
\end{rem}

\end{document}